\documentclass{article}%
\usepackage{amsfonts}
\usepackage{amsmath}%
\usepackage{amssymb}%
\usepackage{graphicx}
\providecommand{\U}[1]{\protect\rule{.1in}{.1in}}
\newtheorem{theorem}{Theorem}

\newtheorem{lemma}[theorem]{Lemma}

\newenvironment{proof}[1][Proof]{\noindent\textbf{#1.} }{\ \rule{0.5em}{0.5em}}
\begin{document}

\title{Intersection numbers for normal functions}
\author{C. Herbert Clemens}
\maketitle

\section{Introduction}

The purpose of this note is to simplify and generalize a formula of M. Green
and P. Griffiths for 'normal functions' in complex algebraic geometry
(\cite{GG}, Thm. 5.5.1). Namely, if I am given the normal functions associated
to two primitive Hodge classes in $H^{2n}\left(  W;\mathbb{Z}\right)  $ for a
complex projective manifold $W$ of (complex) dimension $2n$, how do I
calculate the intersection number of the two primitive classes? It turns out
that this question is entirely topological in nature and its answer has
nothing to do with the complex structures of the objects involved. So we will
begin by describing the notion of \textit{topological normal function}
associated to any primitive element of $H^{2n}\left(  W;\mathbb{Z}\right)  $.
Finally in Theorem \ref{a} we give a formula for computing the intersection
number of two cycles from their topological normal functions. This formula
generalizes the formula of Green and Griffiths and completes and simplifies
their \textquotedblleft\textit{Sketch of the proof:}" in the above-cited
reference. The author would like to thank Mark Green and Christian Schnell for
useful discussions while writing this note, and to especially thank Mirel
Caibar for conversations which led to the formulation and proof of the main theorem.

\section{The topological Poincar\'{e} bundle}

Let $W$ be a complex projective manifold of (complex) dimension $2n$, with
polarizing (very ample) line bundle $\mathcal{O}_{W}\left(  1\right)  $. Thus
there is a natural imbedding%
\begin{gather*}
W\rightarrow\mathbb{P}\left(  H^{0}\left(  \mathcal{O}_{W}\left(  1\right)
\right)  \right) \\
w\mapsto\left\{  \lambda\in H^{0}\left(  \mathcal{O}_{W}\left(  1\right)
\right)  :\lambda\left(  w\right)  =0\right\}  .
\end{gather*}
For the incidence divisor $I\subseteq\mathbb{P}\left(  H^{0}\left(
\mathcal{O}_{W}\left(  1\right)  \right)  ^{\vee}\right)  \times
\mathbb{P}\left(  H^{0}\left(  \mathcal{O}_{W}\left(  1\right)  \right)
\right)  $, define%
\[
X=I\cap\left(  \mathbb{P}\left(  H^{0}\left(  \mathcal{O}_{W}\left(  1\right)
\right)  ^{\vee}\right)  \times W\right)  .
\]
The projection%
\[
\chi:X\rightarrow\mathbb{P}\left(  H^{0}\left(  \mathcal{O}_{W}\left(
1\right)  \right)  ^{\vee}\right)  =:\mathbb{P}%
\]
giving the hyperplane sections of $W$ can be restricted to a smooth morphism%
\[
\chi^{sm.}:X^{sm.}\rightarrow\mathbb{P}^{sm.}%
\]
where is the set of smooth hyperplane sections.

Let $X_{p}$ denote the fiber over $p$. The kernel of the restriction map,
$H^{2n}\left(  W\right)  \rightarrow H^{2n}\left(  X_{p}\right)  $ is
independent of the choice of $p\in\mathbb{P}^{sm.}$ and will be denoted as
$H_{prim.}^{2n}\left(  W\right)  $. A celebrated theorem of Deligne \cite{D}
states that the Leray spectral sequence for $\chi^{sm.}$ degenerates at
$E_{2}$. Let $H^{2n}\left(  X^{sm.}\right)  ^{k}$ denote the filtration of
$H^{2n}\left(  X^{sm.}\right)  $ induced by this spectral sequence, where the
index $k$ indicate the degree of the relevant cochain coming from
$\mathbb{P}^{sm.}$. Thus the preimage $H^{2n}\left(  X^{sm.};\mathbb{Z}%
\right)  ^{1}$under the map%
\[
H^{2n}\left(  W;\mathbb{Z}\right)  \rightarrow H^{2n}\left(  X^{sm.}%
;\mathbb{Z}\right)
\]
is precisely $H_{prim.}^{2n}\left(  W;\mathbb{Z}\right)  $. We will construct
the map%
\begin{equation}
H_{prim.}^{2n}\left(  W;\mathbb{Z}\right)  \rightarrow\frac{H^{2n}\left(
X^{sm.};\mathbb{Z}\right)  ^{1}}{H^{2n}\left(  X^{sm.};\mathbb{Z}\right)
^{2}}=H^{1}\left(  \mathbb{P}^{sm.};R^{2n-1}\chi_{\ast}^{sm.}\mathbb{Z}%
_{X^{sm.}}\right) \label{1}%
\end{equation}
explicitly below, as it will be of central importance to the deriviation of
the formula which is the principal goal of this paper.

Let $L_{X}^{\mathbb{Z}}=$ $\frac{R^{2n-1}\chi_{\ast}^{sm.}\mathbb{Z}_{X^{sm.}%
}}{torsion}$ and analogously define the local system $L_{X}^{\mathbb{Q}%
}=R^{2n-1}\chi_{\ast}^{sm.}\mathbb{Q}_{X^{sm.}}$ and $L_{X}^{\mathbb{R}%
}=R^{2n-1}\chi_{\ast}^{sm.}\mathbb{R}_{X^{sm.}}$. $L_{X}^{\mathbb{R}}$ is the
system of flat sections of the $C^{\infty}$-vector bundle $\mathcal{L}_{X}$
with respect to the Gauss-Manin connection. Define%
\[
J_{X}:=\frac{\mathcal{L}_{X}}{L_{X}^{\mathbb{Z}}}.
\]
The torus bundle $J_{X}$ has a natural dual bundle%
\[
J_{X}^{\vee}:=\frac{Hom\left(  \mathcal{L}_{X},\mathbb{R}\right)  }{Hom\left(
L_{X}^{\mathbb{Z}},\mathbb{Z}\right)  }=\frac{\mathcal{L}_{X}^{\vee}}{\left(
L_{X}^{\mathbb{Z}}\right)  ^{\vee}}.
\]
Via the intersection pairing we have a natural isomorphism%
\[
J_{X}\rightarrow J_{X}^{\vee},
\]
however it will be convenient to keep track of the logical distinction between
these two bundles in what follows. In particular, the identity map
$L_{X}^{\mathbb{Z}}\left(  p\right)  \rightarrow L_{X}^{\mathbb{Z}}\left(
p\right)  $ can be thought of as an element%
\[
\varpi_{p}\in\left(  L_{X}^{\mathbb{Z}}\left(  p\right)  \right)  ^{\vee
}\otimes L_{X}^{\mathbb{Z}}\left(  p\right)  \in H^{2}\left(  J_{X}\left(
p\right)  \times J_{X}^{\vee}\left(  p\right)  ;\mathbb{Z}\right)  .
\]
Thus $\varpi_{p}$ is the first Chern class of a canonical $\mathbb{C}^{\ast}%
$-bundle $\mathcal{P}_{p}$ on $J_{X}\left(  p\right)  \times J_{X}^{\vee
}\left(  p\right)  $ called the Poincar\'{e} bundle.

Now $\mathbb{P}^{sm.}\subseteq\mathbb{P}^{dp.}\subseteq\mathbb{P}$ where
$\left(  \mathbb{P}^{dp.}-\mathbb{P}^{sm.}\right)  $ is the smooth locus of
$\left(  \mathbb{P}-\mathbb{P}^{sm.}\right)  $ and consists of those $p$ such
that $X_{p}$ has a single ordinary node. $\left(  \mathbb{P}-\mathbb{P}%
^{dp.}\right)  $ is of codimension-$2$ in $\mathbb{P}$. The bundle $J_{X}$
extends naturally to a family $\hat{J}_{X}$ of abelian Lie groups over
$\mathbb{P}^{dp.}$as follows:

1) give $J_{X}$ the standard (Griffiths) complex structure so that
$\mathcal{L}_{X}$ becomes $\left(  F^{2n-1}\left(  \mathcal{L}_{X}%
\otimes\mathbb{C}\right)  \right)  ^{\vee}$,

2) extend $\left(  F^{2n-1}\left(  \mathcal{L}_{X}\otimes\mathbb{C}\right)
\right)  ^{\vee}$ over $\left(  \mathbb{P}^{dp.}-\mathbb{P}^{sm.}\right)  $ by
using the standard Deligne extension,

3) let $\mathcal{\hat{L}}_{X}$ denote the Deligne extension, considered as a
\textit{real} vector bundle,

4) take the topological closure $\hat{L}_{X}^{\mathbb{Z}}$ of $L_{X}%
^{\mathbb{Z}}$ considered as a real analytic subvariety of the total space
$\left\vert \mathcal{\hat{L}}_{X}\right\vert $ of $\mathcal{\hat{L}}_{X}$.
(See \cite{S}.)

\noindent For $p\in\left(  \mathbb{P}^{dp.}-\mathbb{P}^{sm.}\right)  $, the
fiber $\hat{L}_{X}^{\mathbb{Z}}\left(  p\right)  $ is a partial lattice of
corank-$1$ in $\mathcal{\hat{L}}_{X}\left(  p\right)  $ and one defines%
\begin{align*}
\hat{J}_{X}\left(  p\right)   & =\frac{\mathcal{\hat{L}}_{X}\left(  p\right)
}{\hat{L}_{X}^{\mathbb{Z}}\left(  p\right)  }\\
\hat{J}_{X}^{\vee}\left(  p\right)   & =\frac{\left(  \mathcal{\hat{L}}%
_{X}\left(  p\right)  \right)  ^{\vee}}{\left(  \hat{L}_{X}^{\mathbb{Z}%
}\left(  p\right)  \right)  ^{\vee}}.
\end{align*}
Also for $p\in\left(  \mathbb{P}^{dp.}-\mathbb{P}^{sm.}\right)  $ define
$X_{p}^{\prime}$ as $X_{p}$ with the node $x_{0}$ removed. We have natural
isomorphisms%
\[
H^{2n-1}\left(  X_{p},\left\{  x_{0}\right\}  ;\mathbb{Z}\right)
\rightarrow\hat{L}_{X}^{\mathbb{Z}}\left(  p\right)
\]
and%
\[
\left(  \hat{L}_{X}^{\mathbb{Z}}\left(  p\right)  \right)  \rightarrow
H^{2n-1}\left(  X_{p}^{\prime};\mathbb{Z}\right)  .
\]
The intersection pairing gives an isomorphism%
\[
H^{2n-1}\left(  X_{p},\left\{  x_{0}\right\}  ;\mathbb{Z}\right)  \rightarrow
H^{2n-1}\left(  X_{p}^{\prime};\mathbb{Z}\right)  ^{\vee}%
\]
and so defines an element%
\[
\varpi_{p}\in\left(  \hat{L}_{X}^{\mathbb{Z}}\left(  p\right)  \right)
^{\vee}\otimes\left(  \hat{L}_{X}^{\mathbb{Z}}\left(  p\right)  \right)
^{\vee}=H^{2}\left(  \hat{J}_{X}\left(  p\right)  \times\hat{J}_{X}^{\vee
}\left(  p\right)  ;\mathbb{Z}\right)  .
\]
The class $\varpi_{p}$ extends the family given by $\varpi_{p}$ at smooth
fibers. (See, for example \cite{GK}, \S 5.)

Also $H^{2}\left(  \mathbb{P}^{dp.};\mathbb{Z}\right)  =H^{2}\left(
\mathbb{P};\mathbb{Z}\right)  =\mathbb{Z}.$ And there is a unique
$\mathbb{C}^{\ast}$-bundle $\mathcal{P}$ on $\hat{J}_{X}\times_{\mathbb{P}%
^{dp.}}\hat{J}_{X}^{\vee}$ whose Chern class $c_{1}\left(  \mathcal{P}\right)
$ restricts to $\varpi_{p}$ on each fiber and whose restriction to the
zero-section of $\tilde{J}_{X}\times_{\mathbb{P}^{dp.}}\tilde{J}_{X}^{\vee}$
is zero. We call $\mathcal{P}$ the \textit{normalized Poincar\'{e} bundle}.

\section{Topological normal functions}

Returning to the map $\left(  \ref{1}\right)  $ the cohomology group
$H^{1}\left(  \mathbb{P}^{sm.};L_{X}^{\mathbb{R}}\right)  $ can be computed
using the deRham resolution%
\[
0\rightarrow L_{X}^{\mathbb{R}}\rightarrow\mathcal{L}_{X}\overset{\nabla
}{\longrightarrow}\mathcal{A}_{\mathbb{P}^{sm.}}^{1}\otimes\mathcal{L}%
_{X}\overset{\nabla}{\longrightarrow}\mathcal{A}_{\mathbb{P}^{sm.}}^{2}%
\otimes\mathcal{L}_{X}\overset{\nabla}{\longrightarrow}\ldots
\]
where $\mathcal{A}_{\mathbb{P}^{sm.}}^{i}$ is the sheaf of $\mathbb{R}$-valued
$i$-forms on $\mathbb{P}^{sm.}$ and the differential is given by the
Gauss-Manin connection. Thus any element $\eta\in H_{prim.}^{2n}\left(
W;\mathbb{Z}\right)  $ gives rise to a closed form $\omega_{\eta}\in
\Gamma\left(  \mathcal{A}_{\mathbb{P}^{sm.}}^{1}\otimes\mathcal{L}_{X}\right)
$ which is well-defined modulo an element of $\nabla\left(  \Gamma
\mathcal{L}_{X}\right)  $. $\omega_{\eta}$ is given by the Cartan-Lie formula,
namely%
\[
\eta\mapsto\omega_{\eta}:=\left(  \zeta\mapsto\left\langle \left.  \nu\left(
\zeta\right)  \right\vert \hat{\eta}\right\rangle \right)
\]
where $\nu\left(  \zeta\right)  \in N_{X_{p}\backslash X^{sm.}}$ is any
lifting of $\zeta\in T_{\mathbb{P}^{sm.},p}$ corresponding to a local
$C^{\infty}$-product structure on $X^{sm.}/\mathbb{P}^{sm.}$ near $p$. We
define an associated \textit{topological normal function} $\alpha_{\eta
}:\mathbb{P}^{sm.}\rightarrow J_{X}$ by the rule%
\[
\alpha_{\eta}\left(  p\right)  =%
{\displaystyle\int\nolimits_{p_{0}}^{p}}
\omega_{\eta}\in J_{X}\left(  p\right)  .
\]

\begin{lemma}
i) Let $B\subseteq\mathbb{P}$ be a smooth complex submanifold with
$\mathrm{dim}_{\mathbb{C}}B\leq n$ that meets $\left(  \mathbb{P}%
^{dp.}-\mathbb{P}^{sm.}\right)  $ transversely. Then $\alpha_{\eta}$ extends
to a section $\hat{\alpha}_{\eta}:B\cap\mathbb{P}^{dp.}\rightarrow\hat{J}%
_{X}\times_{\mathbb{P}}B.$

ii) Suppose that $\eta\in H_{prim.}^{2n}\left(  W;\mathbb{Z}\right)  $ is
strongly primitive, that is, $\left.  \eta\right\vert _{D}=0$ for
\textit{every} codimension-$1$ subvariety $D\subseteq W$. Let $B\subseteq
\mathbb{P}$ be a smooth complex submanifold with $\mathrm{dim}_{\mathbb{C}%
}B\leq2n$ meeting $\left(  \mathbb{P}^{dp.}-\mathbb{P}^{sm.}\right)  $
transversely. Then $\alpha_{\eta}$ extends to a section $\hat{\alpha}_{\eta
}:B\cap\mathbb{P}^{dp.}\rightarrow\hat{J}_{X}\times_{\mathbb{P}}B.$
\end{lemma}

\begin{proof}
\textit{i)} Let $N\subseteq X^{dp.}$ denote the nodal locus. To see that
$\alpha_{\eta}$ extends over $B$, we proceed as follows. Since $X^{dp.}$ is
smooth and $N$ is of real codimension $4n$, the map%
\[
H^{2n}\left(  X^{dp.}\times_{\mathbb{P}}B,N\times_{\mathbb{P}}B;\mathbb{Z}%
\right)  \rightarrow H^{2n}\left(  X^{dp.}\times_{\mathbb{P}}B;\mathbb{Z}%
\right)
\]
is an surjective whenever the complex dimension of $N\times_{\mathbb{P}}B$ is
less than $n$, and an isomorphism whenever its complex dimension is less that
$n-1$. So we can consider the pull-back $\hat{\eta}$ of $\eta$ to
$X^{dp.}\times_{\mathbb{P}}B$ as lifting to an element of $H^{2n}\left(
X^{dp.}\times_{\mathbb{P}}B,N\times_{\mathbb{P}}B;\mathbb{Z}\right)  $
represented by a deRham cycle vanishing in a neighborhood of $N\times
_{\mathbb{P}}B$. Thus for any $\zeta\in T_{\mathbb{P}^{dp.},p}$ for $p$ near a
nodal hypersurface, lift $\zeta$ to a section of the (real) normal and
contract it against $\hat{\eta}$. Since $\hat{\eta}$ vanishes along
$N\times_{\mathbb{P}}B$, this gives a closed form on $\mathbb{P}^{dp.}$ which
extends a realization of $\omega_{\eta}$ and has the property that for any
local section of $T_{\mathbb{P}^{dp.}}$ the corresponding element of $\hat
{L}_{X}^{\mathbb{R}}\left(  p\right)  $ is locally invariant, that is, it lies
in the image of%
\begin{equation}
H^{2n-1}\left(  X_{p},\left\{  x_{0}\right\}  ;\mathbb{R}\right)
\rightarrow\hat{L}_{X}^{\mathbb{R}}\left(  p\right)  .\label{contr}%
\end{equation}

\textit{ii)} By hypothesis, the image of $N\times_{\mathbb{P}}B$ lies in a
codimension-$1$ subvariety $D\subseteq W$. So again the pull-back $\hat{\eta}
$ of $\eta$ to $X^{dp.}\times_{\mathbb{P}}B$ lifts to an element of
$H^{2n}\left(  X^{dp.}\times_{\mathbb{P}}B,N\times_{\mathbb{P}}B;\mathbb{Z}%
\right)  $ represented by a deRham cycle vanishing in a neighborhood of
$N\times_{\mathbb{P}}B$. The rest of the proof is as in \textit{i)}.
\end{proof}

We are finally able to state the formula whose proof in the principal goal of
this paper.

\begin{theorem}
\label{a}Given $\eta^{\prime},\eta^{\prime\prime}\in H_{prim.}^{2n}\left(
W;\mathbb{Z}\right)  $,%
\[%
{\displaystyle\int\nolimits_{W}}
\eta^{\prime}\wedge\eta^{\prime\prime}=\left\langle \left.  \varepsilon
\right\vert c_{1}\left(  s_{\eta^{\prime},\eta^{\prime\prime}}^{\ast
}\mathcal{P}\right)  \right\rangle
\]
where%
\begin{gather*}
s_{\eta^{\prime},\eta^{\prime\prime}}:\mathbb{P}^{dp.}\rightarrow\hat{J}%
_{X}\times_{\mathbb{P}^{dp.}}\hat{J}_{X}^{\vee}\\
p\mapsto\left(  \hat{\alpha}_{\eta^{\prime}},\hat{\alpha}_{\eta^{\prime\prime
}}^{\vee}\right)
\end{gather*}
and $\varepsilon$ is a (positively oriented) generator of $H_{2}\left(
\mathbb{P}^{dp.};\mathbb{Z}\right)  =H_{2}\left(  \mathbb{P};\mathbb{Z}%
\right)  =\mathbb{Z}$.
\end{theorem}

\section{Proof of the theorem}

We may assume that $\varepsilon$ is given by the linear inclusion of a general
projective line into $\mathbb{P}^{dp.}$. That is we reduce to the case of a
Lefschetz pencil%
\[
\chi_{B}:X_{B}\rightarrow B
\]
with local system $L_{X_{B}}$ over $B^{sm.}$. Since $X_{B}\rightarrow W$ is
birational%
\[%
{\displaystyle\int\nolimits_{W}}
\eta^{\prime}\wedge\eta^{\prime\prime}=%
{\displaystyle\int\nolimits_{X_{B}}}
\hat{\eta}^{\prime}\wedge\hat{\eta}^{\prime\prime}=%
{\displaystyle\int\nolimits_{X_{B}^{sm.}}}
\hat{\eta}^{\prime}\wedge\hat{\eta}^{\prime\prime}%
\]
where $\hat{\eta}$ denotes the pull-back of $\eta$ and $X_{B}^{sm.}%
=X_{B}\times_{B}B^{sm.}$. The map associated with the Leray spectral sequence
for $\chi_{B}^{sm.}$ as in $\left(  1\right)  $ is given by putting a local
product structure on $X_{B}$ and writing%
\[
\eta\mapsto\omega_{\eta}:=\left(  \zeta\mapsto\left\langle \left.  \nu\left(
\zeta\right)  \right\vert \hat{\eta}\right\rangle \right)
\]
where $\nu\left(  \zeta\right)  \in T_{X_{B}\left(  p\right)  }$ is the
lifting of $\zeta\in T_{B^{sm.},p}$ given by the product structure. Then it is
immediate that $\omega_{\eta}$considered as an element of $\mathcal{A}%
_{\mathbb{P}^{sm.}}^{1}\otimes\mathcal{L}_{X}$ is independent of the product
structure and direct computation shows that $\nabla\omega_{\eta}=0$.

Next, following Lefschetz, we pick a basepoint $p_{\infty}\in B^{sm.}$ and
draw non-intersecting paths from $p_{\infty}$ to the points of $B-B^{sm.}$. We
let $U\subseteq B$ denote the simply connected region comprising the
complement of the union of those paths. Then for any $p_{0}\in U$,%
\[
X_{U}:=X_{B}\times_{B}U
\]
is isomorphic over $U$ to $X_{p_{0}}\times U$. And by the Fubini theorem%
\begin{align}%
{\displaystyle\int\nolimits_{X_{B}^{sm.}}}
\hat{\eta}^{\prime}\wedge\hat{\eta}^{\prime\prime}  & =%
{\displaystyle\int\nolimits_{U}}
\hat{\eta}^{\prime}\wedge\hat{\eta}^{\prime\prime}\nonumber\\
& =%
{\displaystyle\int\nolimits_{U}}
\left(  \zeta^{\prime}\wedge\zeta^{\prime\prime}\mapsto%
{\displaystyle\int\nolimits_{X_{p}}}
\left\langle \left.  \nu\left(  \zeta^{\prime}\right)  \right\vert \hat{\eta
}^{\prime}\right\rangle \wedge\left\langle \left.  \nu\left(  \zeta
^{\prime\prime}\right)  \right\vert \hat{\eta}^{\prime\prime}\right\rangle
\right) \label{2}\\
& =%
{\displaystyle\int\nolimits_{U}}
\left\langle \left.  \omega_{\eta^{\prime}}\right\vert \omega_{\eta
^{\prime\prime}}^{\vee}\right\rangle .\nonumber
\end{align}

But now choose a symplectic framing $\left\{  \alpha_{j},\beta_{j}\right\}  $
of $L_{X_{B}}^{\mathbb{Z}}$ over $U$ and write%
\begin{align*}
\omega_{\eta^{\prime}}  & =%
{\displaystyle\sum\nolimits_{i=1,2}}
\left(
{\displaystyle\sum\nolimits_{j}}
f_{j,i}^{\prime}\alpha_{j}+%
{\displaystyle\sum\nolimits_{j}}
g_{j,i}^{\prime}\beta_{j}\right)  dx_{i}\\
\omega_{\eta^{\prime\prime}}  & =%
{\displaystyle\sum\nolimits_{i=1,2}}
\left(
{\displaystyle\sum\nolimits_{j}}
f_{j,i}^{\prime\prime}\alpha_{j}+%
{\displaystyle\sum\nolimits_{j}}
g_{j,i}^{\prime\prime}\beta_{j}\right)  dx_{i}.
\end{align*}
The integral in $\left(  \ref{2}\right)  $ becomes%
\[%
{\displaystyle\int\nolimits_{U}}
\left(  \left(  f_{j}^{\prime}\right)  \left(  g_{j}^{\prime}\right)  \right)
\left(
\begin{array}
[c]{cc}%
0 & I\\
-I & 0
\end{array}
\right)  \left(
\begin{array}
[c]{c}%
\left(  f_{j}^{\prime\prime}\right) \\
\left(  g_{j}^{\prime\prime}\right)
\end{array}
\right)  dx_{1}dx_{2}.
\]

On the other hand, over $U$ we can write%
\begin{align*}
\alpha_{\eta^{\prime}}\left(  p\right)   & =%
{\displaystyle\sum\nolimits_{j}}
{\displaystyle\int\nolimits_{p_{0}}^{p}}
\left(
{\displaystyle\sum\nolimits_{i=1,2}}
f_{j,i}^{\prime}dx_{i}\right)  \alpha_{j}+%
{\displaystyle\sum\nolimits_{j}}
\left(
{\displaystyle\sum\nolimits_{i=1,2}}
g_{j,i}^{\prime}dx_{i}\right)  \beta_{j}\\
\alpha_{\eta^{\prime\prime}}\left(  p\right)   & =%
{\displaystyle\sum\nolimits_{j}}
{\displaystyle\int\nolimits_{p_{0}}^{p}}
\left(
{\displaystyle\sum\nolimits_{i=1,2}}
f_{j,i}^{\prime\prime}dx_{i}\right)  \alpha_{j}+%
{\displaystyle\sum\nolimits_{j}}
\left(
{\displaystyle\sum\nolimits_{i=1,2}}
g_{j,i}^{\prime\prime}dx_{i}\right)  \beta_{j}%
\end{align*}
and notice that the isomorphism $J_{X}\rightarrow J_{X}^{\vee}$ is given, in
terms of the above framing, exactly by multiplication by the matrix $\left(
\begin{array}
[c]{cc}%
0 & I\\
-I & 0
\end{array}
\right)  $. So $c_{1}\left(  \mathcal{P}\right)  $ is computed on the graph%
\[
\left(  \alpha_{\eta^{\prime}},\alpha_{\eta^{\prime\prime}}^{\vee}\right)
\]
by the integral%
\[%
{\displaystyle\int\nolimits_{U}}
d\alpha_{\eta^{\prime}}\left(
\begin{array}
[c]{cc}%
0 & I\\
-I & 0
\end{array}
\right)  d\alpha_{\eta^{\prime\prime}}.
\]

\section{Relationship with classical normal functions}

Suppose now that $\eta\in H^{2n}\left(  W;\mathbb{Z}\right)  \cap
F^{2n-k}H^{2n}\left(  W;\mathbb{C}\right)  $. Since $L_{X}^{\mathbb{C}}$ is
the sheaf of flat sections of a holomorphic vector bundle $\mathcal{L}_{X}$,
we have the resolution of $L_{X}^{\mathbb{C}}\rightarrow\mathcal{L}%
_{X}^{\mathbb{C}}$ and a morphism of complexes of fine sheaves%
\begin{equation}%
\begin{array}
[c]{ccccccc}%
\mathcal{A}^{0}\left(  _{X}^{\mathbb{R}}\right)  & \overset{\nabla
}{\longrightarrow} & \mathcal{A}^{1}\left(  L_{X}^{\mathbb{R}}\right)  &
\overset{\nabla}{\longrightarrow} & \mathcal{A}^{2}\left(  L_{X}^{\mathbb{R}%
}\right)  & \overset{\nabla}{\longrightarrow} & \ldots\\
\downarrow &  & \downarrow &  & \downarrow &  & \\
\mathcal{A}^{0,0}\left(  \left(  F^{k}\mathcal{L}_{X}^{\mathbb{C}}\right)
^{\vee}\right)  & \overset{\nabla^{0,1}}{\longrightarrow} & \mathcal{A}%
^{0,1}\left(  \left(  F^{k}\mathcal{L}_{X}^{\mathbb{C}}\right)  ^{\vee}\right)
& \overset{\nabla^{0,1}}{\longrightarrow} & \mathcal{A}^{0,2}\left(  \left(
F^{k}\mathcal{L}_{X}^{\mathbb{C}}\right)  ^{\vee}\right)  & \overset
{\nabla^{0,1}}{\longrightarrow} & \ldots
\end{array}
\label{Hcom}%
\end{equation}
where $F^{\ast}\mathcal{L}_{X}^{\mathbb{C}}$ denotes the Hodge filtration and
the mapping
\[
L_{X}^{\mathbb{R}}\rightarrow\left(  F^{k}\mathcal{L}_{X}^{\mathbb{C}}\right)
^{\vee}%
\]
is given by the intersection pairing. We let $L_{X,van.}^{\mathbb{Q}}\subseteq
L_{X}^{\mathbb{Q}}$ denote the sub-local-system whose fiber is the image of
the residue map $H^{2n}\left(  W-X_{p}\right)  \rightarrow H^{2n-1}\left(
X_{p}\right)  $ , $L_{W}^{\mathbb{Q}}$ the (trivial) local system with fiber
$H_{prim.}^{2n}\left(  W;\mathbb{Q}\right)  $, and $L_{W|X}^{\mathbb{Q}}$ the
local system with fiber $H^{2n}\left(  W,X_{p};\mathbb{Q}\right)  $. Thus
$\eta$ can be considered as an element of the Deligne cohomology group%
\[
\mathbb{H}^{0}\left(  L_{X}^{\mathbb{Q}}\rightarrow\left(  F^{k+1}%
\mathcal{L}_{W}^{\mathbb{C}}\right)  ^{\vee}\right)  .
\]
Let $L_{X,van.}^{\mathbb{Q}}$ denote the vanishing cohomology. Then, under the
connection homomorphism associated to the short exact sequence of complexes%
\[%
\begin{array}
[c]{ccccc}%
L_{X,van.}^{\mathbb{Q}} & \rightarrow & L_{W|X}^{\mathbb{Q}} & \rightarrow &
L_{W}^{\mathbb{Q}}\\
\downarrow &  & \downarrow &  & \downarrow\\
\left(  F^{k}\mathcal{L}_{X,van.}^{\mathbb{C}}\right)  ^{\vee} & \rightarrow &
\left(  F^{k+1}\mathcal{L}_{W|X}^{\mathbb{C}}\right)  ^{\vee} & \rightarrow &
\left(  F^{k+1}\mathcal{L}_{W}^{\mathbb{C}}\right)  ^{\vee}%
\end{array}
\]
and the inclusion map%
\[%
\begin{array}
[c]{ccc}%
L_{X,van.}^{\mathbb{Q}} & \rightarrow & L_{X}^{\mathbb{Q}}\\
\downarrow &  & \downarrow\\
\left(  F^{k}\mathcal{L}_{X,van.}^{\mathbb{C}}\right)  ^{\vee} & \rightarrow &
\left(  F^{k}\mathcal{L}_{X}^{\mathbb{C}}\right)  ^{\vee},
\end{array}
\]
$\eta$ induces an element $\left\{  \omega_{\eta}^{k}\right\}  $ of the
Deligne cohomology group%
\[
\mathbb{H}^{1}\left(  L_{X}^{\mathbb{Q}}\rightarrow\left(  F^{k}%
\mathcal{L}_{X}^{\mathbb{C}}\right)  ^{\vee}\right)
\]
which maps to $\left\{  \omega_{\eta}\right\}  $ under the natural map%
\[
\mathbb{H}^{1}\left(  L_{X}^{\mathbb{Q}}\rightarrow\left(  F^{k}%
\mathcal{L}_{X}^{\mathbb{C}}\right)  ^{\vee}\right)  \rightarrow H^{1}\left(
L_{X}^{\mathbb{Q}}\right)  .
\]

Via the resolutions $\left(  \ref{Hcom}\right)  $ therefore, $\left\{
\omega_{\eta}^{k}\right\}  $ can be represented by a form $\omega_{\eta}^{k}$
which maps to zero in $\Gamma\mathcal{A}^{0,1}\left(  \left(  F^{k}%
\mathcal{L}_{X}\right)  ^{\vee}\right)  $, that is, by a holomorphic one-form
with coefficients in $\left(  F^{k}\mathcal{L}_{X}^{\mathbb{C}}\right)
^{\vee}$. Picking a basepoint $p_{0}\in\mathbb{P}^{sm.}$, we define a
$k$\textit{-normal function} as the multi-valued function%
\begin{equation}
\nu_{\eta}^{k}\left(  p\right)  :=%
{\displaystyle\int\nolimits_{p_{0}}^{p}}
\omega_{\eta}^{k}\in\left\vert \left(  F^{k}\mathcal{L}_{X}^{\mathbb{C}%
}\right)  ^{\vee}\right\vert .\label{nlfn}%
\end{equation}
The classical normal function is simply an $n$-normal function (with
logarithmic growth at $\left(  \mathbb{P}^{dp.}-\mathbb{P}^{sm.}\right)  $).

\end{document}